\documentclass[11pt,twoside]{amsart}
\usepackage{mathrsfs}
\usepackage{amsmath}
\usepackage{amsthm}
\usepackage{amsfonts}
\usepackage{amssymb}
\usepackage{latexsym}
\usepackage[all]{xy}
\usepackage{extarrows}
\usepackage{graphicx}

\date{\empty}
\thanks{}
\allowdisplaybreaks[4] \footskip=15pt
\renewcommand{\uppercasenonmath}[1]{}

\numberwithin{equation}{section} \theoremstyle{plain}
\newtheorem*{thm*}{Main Theorem}
\newtheorem{theorem}{Theorem}[section]
\newtheorem{corollary}[theorem]{Corollary}
\newtheorem*{corollary*}{Corollary}

\newtheorem*{claim*}{Claim}
\newtheorem{lemma}[theorem]{Lemma}
\newtheorem*{lemma*}{Lemma}
\newtheorem{proposition}[theorem]{Proposition}
\newtheorem*{proposition*}{Proposition}
\newtheorem{remark}[theorem]{Remark}
\newtheorem*{remark*}{Remark}

\newtheorem*{example*}{Example}

\newtheorem*{question*}{Question}
\newtheorem{definition}[theorem]{Definition}
\newtheorem*{definition*}{Definition}

\newtheorem*{acknowledgements*}{ACKNOWLEDGEMENTS}

\def\sbmatrix{\left[\begin{array}}
\def\endsbmatrix{\end{array}\right]}

\begin{document}
\begin{center}
{\large  \bf Right core inverse and the related generalized inverses}\\
\vspace{0.8cm}   Long Wang$^{1}$, Dijana Mosi\'{c}$^{2}$ and Yuefeng Gao$^{3}$\footnote{Corresponding author.
The research is supported by the NSFC (11771076), NSF of Jiangsu Province (BK20170589),
NSF of Jiangsu Higher Education Institutions of China (15KJB110021). The second author is supported by the Ministry of Science,
Republic of Serbia, grant no. 174007.}\\
\vspace{0.5cm} { $^{1}$School of Mathematics, Yangzhou University, Yangzhou, 225002, P. R. China\\
$^{2}$Faculty of Sciences and Mathematics, University of Ni\v{s}, Ni\v{s}, Serbia\\
$^{3}$School of Mathematics, Southeast University, Nanjing, 210096, P. R. China}
\end{center}

\bigskip

{ \bf Abstract:}  \leftskip0truemm\rightskip0truemm
In this paper, we introduce the notion of a  (generalized) right core inverse and
give  its characterizations and expressions.
Then, we  provide the relation schema of  (one-sided) core inverses,  (one-sided) pseudo core inverses and EP elements.
\\{  \textbf{Keywords:}} Core inverse; Right core inverse; Pseudo core inverse; Ring.
\\\noindent { \textbf{2010 Mathematics Subject Classification:}} 15A09; 16W10.
 \bigskip

%%%%%%%%%%%%%%%%%%%%%%%%%%%%%%%%%%%%%%%%%%%%%%%%%%%%%%%%%%%%%%%%%%%%%
%%%%%%%%%%%%%%%%%%%%%    Section 1   %%%%%%%%%%%%%%%%%%%%%%%%%%%%
%%%%%%%%%%%%%%%%%%%%%%%%%%%%%%%%%%%%%%%%%%%%%%%%%%%%%%%%%%%%%%%%%%%%%

\section{Introduction }
Throughout this paper, $\mathcal{R}$ is a ring with identity 1.
Recall that an element $a \in \mathcal{R}$ is (von Neumann)
regular if there exists $x \in \mathcal{R}$ satisfying $axa = a$.
Such $x$ is called an inner inverse of $a$ and denoted by $a^{-}$.
For $a\in \mathcal{R}$, if $xax=x$ holds for some $x\in \mathcal{R}\backslash\{0\}$,
then $x$ is an outer generalized inverse of $a$.
The Drazin inverse \cite{MPD3} of $a \in \mathcal{R}$ is the element $x \in \mathcal{R}$ which satisfies
\begin{center}
$xax = x$,\quad $ax = xa$,\quad $a^{k} = a^{k+1}x$\quad for some $k \geq 0$.
\end{center}
The element $x$ above is unique if it exists and is denoted by $a^{D}$.
The least such $k$ is called the index of $a$, and denoted by $\text{ind}(a)$.
In particular, when $\text{ind}(a)=1$, the Drazin inverse $a^{D}$ is called the group inverse of $a$ and it is denoted by $a^{\sharp}$.
The set of all Drazin (resp. group) invertible elements of $\mathcal{R}$ is denoted by $\mathcal{R}^{D}$ (resp. $\mathcal{R}^{\sharp}$).

In a ring $\mathcal{R}$, an involution $\ast: \mathcal{R} \rightarrow \mathcal{R}$
is an anti-isomorphism which satisfies $(a^{\ast})^{\ast} = a$, $(a + b)^{\ast} = a^{\ast} + b^{\ast}$ and
$(ab)^{\ast}= b^{\ast}a^{\ast}$ for all $a, b \in \mathcal{R}$. $\mathcal{R}$ is called a $\ast$-ring if $R$ is a ring with involution.
In what follows, $\mathcal{R}$ is a $\ast$-ring. An element $a \in \mathcal{R}$ is said to be Moore-Penrose invertible if the following equations:
\begin{center}
(1) $axa = a$, (2) $xax = x$, (3) $(ax)^{\ast} = ax$, (4) $(xa)^{\ast} = xa$
\end{center}
have a common solution \cite{R.P}. Such solution is unique if it exists, and is usually
denoted by $a^{\dag}$. The set of all Moore-Penrose invertible elements of $\mathcal{R}$ will be
denoted by $\mathcal{R}^{\dag}$.
If $x \in \mathcal{R}$ satisfies both Eqs. (1) and (3), then $x$ is called  an $\{1, 3\}$-inverse of $a$ and denoted by $a^{(1,3)}$.
The set of all $\{1, 3\}$-invertible elements of $\mathcal{R}$ is denoted by $\mathcal{R}^{\{1,3\}}$.
Similarly, if $x \in \mathcal{R}$ satisfies both Eqs. (1) and (4), then $x$ is called an $\{1, 4\}$-inverse of $a$ and denoted by $a^{(1,4)}$.
The set of all $\{1, 4\}$-invertible elements of $\mathcal{R}$ is denoted by $\mathcal{R}^{\{1,4\}}$.

In 2010, Baksalary and Trenkler \cite{OMB1} introduced the concept of core inverse for complex matrices.
Later, Raki\'{c} et al. \cite{DSR1} extended this concept to an arbitrary $\ast$-ring case.
The core (resp. dual core) inverse \cite{DSR1} of $a \in \mathcal{R}$ is the element $x \in \mathcal{R}$ which satisfies
\begin{center}
$axa = a$, $x\mathcal{R} = a\mathcal{R}$ (resp. $\mathcal{R}x = \mathcal{R}a$), $\mathcal{R}x= \mathcal{R}a^{\ast}$ (resp. $x\mathcal{R} = a^{\ast}\mathcal{R}$).
\end{center}
The element $x$ above is unique if it exists and is denoted by $a^{\tiny{\textcircled{\tiny\#}}}$(resp. $a_{\tiny{\textcircled{\tiny\#}}}$).
The set of all core (resp. dual core) invertible elements of $\mathcal{R}$ will be denoted
by $R^{\tiny{\textcircled{\tiny\#}}}$ (resp. $R_{\tiny{\textcircled{\tiny\#}}}$).
They also proved that the core inverse of $a$ is the unique element $a^{\tiny{\textcircled{\tiny\#}}}$ satisfying the following five equations
\begin{center}
$aa^{\tiny{\textcircled{\tiny\#}}}a = a$, $a^{\tiny{\textcircled{\tiny\#}}}aa^{\tiny{\textcircled{\tiny\#}}} = a^{\tiny{\textcircled{\tiny\#}}}$,
$(aa^{\tiny{\textcircled{\tiny\#}}})^{\ast} = aa^{\tiny{\textcircled{\tiny\#}}}$, $a^{\tiny{\textcircled{\tiny\#}}}a^{2} = a$,
$a(a^{\tiny{\textcircled{\tiny\#}}})^{2 }= a^{\tiny{\textcircled{\tiny\#}}}$.
\end{center}
Recently, Xu, Chen and Zhang \cite{SZX1} characterized the core invertible elements in $\mathcal{R}$ in terms of three equations.
The core inverse of $a$ is the unique solution to equations
\begin{center}
$xa^{2} = a$, $ax^{2} = x$ and $(ax)^{\ast} = ax^{\ast}$.
\end{center}
Further, they pointed out that $a \in \mathcal{R}^{\tiny{\textcircled{\tiny\#}}}$ if and only if $a \in \mathcal{R}^{\sharp}$ and $a \in \mathcal{R}^{\{1,3\}}$,
in which case, $a^{\tiny{\textcircled{\tiny\#}}} = a^{\sharp}aa^{(1,3)}$.
$p \in \mathcal{R}$ is a projection if  $p^{2} = p = p^{\ast}$.
In \cite{LC}, Li and Chen gave the characterizations and expressions of core inverse of an element by projections and units.
They proved that $a$ is core invertible if and only if there exists a projection $p$
such that $pa=0$, $a^{n}+p$ is invertible for $n \geq 1$.
In this article, we will consider the case that $a^{n}+p$ is one-sided invertible,
and then the concepts of (generalized) one-sided core inverses are introduced.
It is worth mentioning that in \cite{LC} the authors proved that
$a \in \mathcal{R}(a^{\ast})^{n}a$ is equivalent to  $a \in \mathcal{R}a^{\ast}a \cap a^{n}\mathcal{R}$ for $n \geq 1$.
In this article, we will prove that this condition is also a characterization of right core inverse of $a$.

Now, we give the main concepts  and symbols.
Drazin in \cite[Definition 1.3]{MPD1} introduced the following outer generalized inverse,
called it $(b,c)$-inverse:
\begin{definition}
Given any $a, b, c, y \in \mathcal{R}$, then $y$ is called a $(b, c)$-inverse of $a$ if
\begin{center}
$y\in (b\mathcal{R}y) \cap (y\mathcal{R}c)$, $yab = b$ and $cay = c$.
\end{center}
\end{definition}
Any such $y$ is unique if it exists, in which case also $yay=y$.
We will use $a^{||(b,c)}$ to denote the $(b,c)$-inverse of $a$.
It is well known that the Drazin inverse of $a$ is the $(a^{j}, a^{j})$-inverse of $a$ for
some $j \in \mathbb{N}$, in particular, the group inverse of $a$ is the $(a, a)$-inverse of $a$.
The Moore-Penrose inverse of $a$ is the $(a^{\ast}, a^{\ast})$-inverse of $a$.
The core inverse of $a$ is the $(a,a^{\ast})$-inverse of $a$,
which is a special case of $(b,c)$-inverse where in general $b\neq c$.
The dual core inverse of $a$ is the $(a^{\ast}, a)$-inverse of $a$.
In \cite[Theorem 2.2]{MPD1}, Drazin proved that, for any given $a,  b, c \in \mathcal{R}$, there exists the $(b, c)$-inverse $y$
of $a$ if and only if $\mathcal{R}b = \mathcal{R}cab$ and $c\mathcal{R} = cab\mathcal{R}$.
Later, in \cite{WCN}, the authors shown that if the $(b, c)$-inverse of $a$ exists,
then $b$, $c$ and $cab$ are regular.
Further, under the hypothesis of $cab$ regular, some characterizations
of the $(b, c)$-inverse are obtained.
In particular, It is proven that $(b, c)$-inverse,
hybrid $(b, c)$-inverse, and annihilator $(b, c)$-inverse are coincident when $cab$ is regular.
As weaker versions of the $(b,c)$-invertibility, one-sided $(b,c)$-invertibility is introduced by Drazin \cite{MPD2}:
\begin{definition}
Let $a, b, c\in \mathcal{R}$. Then $a$ is left $(b,c)$-invertible if $b\in \mathcal{R}cab$,
or equivalently if there exists $x\in \mathcal{R}c$ such that $xab=b$,
in which case any such $x$ will be called a left $(b,c)$-inverse of $a$.
\end{definition}
Dually, $a$ is called right $(b,c)$-invertible if $c\in cab\mathcal{R}$,
or equivalently if there exists $z\in b\mathcal{R}$ such that $caz=c$,
in which case any such $z$ will be called a right $(b,c)$-inverse of $a$.

In \cite{MPD2}, Drazin considers some properties of left (or right) $(b, c)$-inverses under the additional conditions,
such as $\mathcal{R}$ is strongly $\pi$-regular.
In \cite{WD}, the authors continue to study the properties of left (or right) $(b, c)$-inverses under the condition $cab$ is regular.
As applications, the authors introduced the one-sided core inverse, and provide a characterization of right core inverses \cite[Theorem 5.1]{WD}.
For the convenience of the reader, the definitions of right core inverses are given again.
\begin{definition}
Let $a\in \mathcal{R}$. We say that $a$ is right core invertible if $a$ is right $(a,a^*)$-invertible.
\end{definition}

Recall that $a$ is right core invertible if and only if $a^*\in a^*a^2\mathcal{R}$
if and only if there exists $x\in \mathcal{R}$ such that $x\in a\mathcal{R}$ and $a^*ax=a^*$.
The sets of all right core invertible elements of $\mathcal{R}$ will be denoted by $\mathcal{R}^{\tiny{\textcircled{\tiny\#}}}_{r}$.
The symbol $a^{\tiny{\textcircled{\tiny\#}}}_{r}$ is used to denote the right core inverse of $a$, when $a\in \mathcal{R}^{\tiny{\textcircled{\tiny\#}}}_{r}$.

Motivated by the results of \cite{LC,WD}, in Section 2 and Section 3,
we will continue to study the characterization of right core inverses of $a$.
In Section 4, we will provide the relation schema of several kinds of (generalized one-sided) core inverses.
%%%%%%%%%%%%%%%%%%%%%%%%%%%%%%%%%%%%%%%%%%%%%%%%%%%%%%%%%%%%%%%%%%%%%
%%%%%%%%%%%%%%%%%%%%%    Section 2   %%%%%%%%%%%%%%%%%%%%%%%%%%%%
%%%%%%%%%%%%%%%%%%%%%%%%%%%%%%%%%%%%%%%%%%%%%%%%%%%%%%%%%%%%%%%%%%%%%

\section{ Characterizing right core inverses by
projections and units in a $\ast$-ring}

In \cite[Theorem 3.3 and Theorem 3.4]{LC},
the authors gave the characterizations and expressions of core inverse of an element by a projection and units.
In this section, we present some equivalent conditions for the existence of right core inverses.
It is proven that $a$ is right core invertible if and only if there exists a projection $p$
such that $pa=0$, $a^{n}+p$ is right invertible for $n \geq 1$.
Before we start, look at the following results.

\begin{lemma} \label{lem21}
\cite[Theorem 5.1]{WD} Let $a\in \mathcal{R}$. Then the following statements are equivalent:

$(i)$ $a$ is right core invertible;

$(ii)$ there exists $x\in \mathcal{R}$ such that $axa=a$, $x=ax^{2}$ and $(ax)^{\ast}=ax$.
\end{lemma}

\begin{remark}\label{rem22}
Checking the proof of Lemma \ref{lem21} \cite[Theorem 5.1]{WD},
we can find that all right core inverses $a^{\tiny{\textcircled{\tiny\#}}}_{r}$ of $a$ satisfy
\begin{center}
$aa^{\tiny{\textcircled{\tiny\#}}}_{r}a=a$, $a^{\tiny{\textcircled{\tiny\#}}}_{r}=a(a^{\tiny{\textcircled{\tiny\#}}}_{r})^{2}$
and $(aa^{\tiny{\textcircled{\tiny\#}}}_{r})^{\ast}=aa^{\tiny{\textcircled{\tiny\#}}}_{r}$.
\end{center}
Moreover, if $a$ is right core invertible, then $aa^{\tiny{\textcircled{\tiny\#}}}_{r}$ is invariant on the choice of $a^{\tiny{\textcircled{\tiny\#}}}_{r}$.
Indeed, assume that $x_{1}$ and $x_{2}$ are two right core inverses of $a$. Then
$ax_{1}=(ax_{1})^{\ast}=x_{1}^{\ast}a^{\ast}=x_{1}^{\ast}a^{\ast}ax_{2}=ax_{2}$.
We say that $a^{\pi}=1-aa^{\tiny{\textcircled{\tiny\#}}}_{r}$ is the spectral idempotents of $a$ when $a$ is right core invertible in $\mathcal{R}$.
\end{remark}

An element $a \in \mathcal{R}$ is called EP if $a \in \mathcal{R}^{\dag}$ and $aa^{\dag} = a^{\dag}a$.
It is well known that $a$ is EP if and only if $a\in \mathcal{R}^{\dag} \cap \mathcal{R}^{\sharp}$ and $a^{\sharp}=a^{\dag}$.
In \cite{SZX1}, the authors pointed out that $a \in \mathcal{R}^{\tiny{\textcircled{\tiny\#}}}$
if and only if $a \in \mathcal{R}^{\sharp} \cap \mathcal{R}^{\{1,3\}}$. By Lemma \ref{lem21} or Remark \ref{rem22}, we deduce the following:

\begin{corollary}
Let $a\in \mathcal{R}$ be right core invertible.
If $aa^{\tiny\textcircled{\tiny\#}}_r=a^{\tiny\textcircled{\tiny\#}}_ra$,
then $a$ is EP and $a^{\tiny\textcircled{\tiny\#}}_r=a^\#=a^\dag$.
\end{corollary}

Under the hypothesis $x=ax^{2}$, which implies $x=a^kx^{k+1}$, for any $k\geq 1$,
we deduce that $axa=a$ and $(ax)^{\ast}=ax$ if and only if $a^{k+1}x^{k+1}a=a$ and $(a^{k+1}x^{k+1})^{\ast}=a^{k+1}x^{k+1}$.
Thus, we obtain new characterization of right core invertible elements.

\begin{corollary}
Let $a\in \mathcal{R}$. Then the following statements are equivalent:

$(i)$ $a$ is right core invertible;

$(ii)$ there exists $x\in \mathcal{R}$ such that $a^{k+1}x^{k+1}a=a$, $x=ax^{2}$ and $(a^{k+1}x^{k+1})^{\ast}=a^{k+1}x^{k+1}$, for any $k\geq 1$;

$(iii)$ there exists $x\in \mathcal{R}$ such that $a^{k+1}x^{k+1}a=a$, $x=ax^{2}$ and $(a^{k+1}x^{k+1})^{\ast}=a^{k+1}x^{k+1}$, for some $k\geq 1$.
\end{corollary}

In the following, we will use the symbol $\mathcal{R}^{-1}_{r}$ to denote the set of all right invertible elements in $\mathcal{R}$.
The symbol $a^{-1}_{r}$ denotes the right inverse of $a$, when $a\in \mathcal{R}^{-1}_{r}$.
The symbol $r(a)$ (rep. $l(a)$)  denotes the right (rep. left) annihilator of $a\in \mathcal{R}$.

\begin{theorem}\label{the25}
Let  $a\in \mathcal{R}$. Then the following statements are equivalent:

$(i)$ $a$ is right core invertible;

$(ii)$ there exists a unique projection $p$ such that $pa=0$ and $u=p+a\in \mathcal{R}^{-1}_{r}$;

$(iii)$ there exists a unique projection $p$ such that $pa=0$ and $w=a(1-p)+p\in \mathcal{R}^{-1}_{r}$.

In this case, $a^{\tiny{\textcircled{\tiny\#}}}_{r}=u^{-1}_{r}(1-p)=(1-p)w^{-1}_{r}.$
\end{theorem}

\begin{proof}
$(i) \Rightarrow (ii)$ Set $p=1-aa^{\tiny{\textcircled{\tiny\#}}}_{r}$.
From Remark \ref{rem22},  we know that $p$ is a projection and $pa=0$.
Note that $pa^{\tiny{\textcircled{\tiny\#}}}_{r}=pa(a^{\tiny{\textcircled{\tiny\#}}}_{r})^{2}=0$, then
\begin{center}
$(p+a)(a^{\tiny{\textcircled{\tiny\#}}}_{r}+1-a^{\tiny{\textcircled{\tiny\#}}}_{r}a)=p+aa^{\tiny{\textcircled{\tiny\#}}}_{r}=1$.
\end{center}
This means $p+a$ is right invertible in $\mathcal{R}$.

For the uniqueness of the projection, we can assume that there exist two projections $p$ and $q$ satisfy
$pa=qa=0$, $p+a$ and $q+a$ are right invertible in $\mathcal{R}$. Then it is easy to get $l(1-p)=l(1-q)=l(a)$. Indeed,
$(1-p)a=a$ gives that $l(1-p)\subseteq l(a)$. And $(1-p)(p+a)=a$, it gives that $1-p=a(p+a)^{-1}_{r}$ and $l(a) \subseteq l(1-p)$.
Similarly, we can get $l(1-q)=l(a)$. Then, $p\in l(1-p)=l(1-q)$ and $q\in l(1-q)=l(1-p)$, this implies that $p=pq$ and $q=qp$.
Hence, $p=pq=(pq)^{\ast}=qp=q$.

$(ii) \Rightarrow (i)$ By hypothesis $pa=0$ and $p+a\in \mathcal{R}^{-1}_{r}$, we have $(1-p)(p+a)=a$ and $1-p=a(p+a)^{-1}_{r}$. Set $x=(p+a)^{-1}_{r}(1-p)$. Then we obtain
\begin{center}
$ax=a(p+a)^{-1}_{r}(1-p)=1-p$,
\end{center}
it gives that $(ax)^{\ast}=ax$. It is easy to get $axa=(1-p)a=a$.
Note that $p(p+a)=p$, then $p=p(p+a)^{-1}_{r}$ and $(1-p)(p+a)^{-1}_{r}=(p+a)^{-1}_{r}-p$.
Hence,
\begin{center}
$ax^{2}=(1-p)(p+a)^{-1}_{r}(1-p)=[(p+a)^{-1}_{r}-p](1-p)=(p+a)^{-1}_{r}(1-p)=x$.
\end{center}
By Lemma \ref{lem21}, we obtain $a$ is right core invertible.

$(i) \Rightarrow (iii)$ We also set $p=1-aa^{\tiny{\textcircled{\tiny\#}}}_{r}$ and $v=a^{\tiny{\textcircled{\tiny\#}}}_{r}+1-aa^{\tiny{\textcircled{\tiny\#}}}_{r}$.
Then $pa=0$ and
\begin{eqnarray*}
[a(1-p)+p]v&=&a^{2}a^{\tiny{\textcircled{\tiny\#}}}_{r}(a^{\tiny{\textcircled{\tiny\#}}}_{r}+1-aa^{\tiny{\textcircled{\tiny\#}}}_{r})+1-aa^{\tiny{\textcircled{\tiny\#}}}_{r}\\
&=&a^{2}(a^{\tiny{\textcircled{\tiny\#}}}_{r})^{2}+a^{2}a^{\tiny{\textcircled{\tiny\#}}}_{r}
-a^{2}a^{\tiny{\textcircled{\tiny\#}}}_{r}aa^{\tiny{\textcircled{\tiny\#}}}_{r}+1-aa^{\tiny{\textcircled{\tiny\#}}}_{r}\\
&=&aa^{\tiny{\textcircled{\tiny\#}}}_{r}+a^{2}a^{\tiny{\textcircled{\tiny\#}}}_{r}-a^{2}a^{\tiny{\textcircled{\tiny\#}}}_{r}+1-aa^{\tiny{\textcircled{\tiny\#}}}_{r}\\
&=&1
\end{eqnarray*}
For the uniqueness of the projection, it is similar to $(i) \Rightarrow (ii)$.

$(iii) \Rightarrow (i)$ Note that $(1-p)w=(1-p)[a(1-p)+p]=(1-p)a(1-p)=a(1-p)$. Then $1-p=a(1-p)w^{-1}_{r}$.
We set $x=(1-p)w^{-1}_{r}$. It is clear that $ax=1-p=(ax)^{\ast}$. And $axa=(1-p)a=a$ and
$ax^{2}=(1-p)x=x$. Thus, by Lemma \ref{lem21}, we obtain $a$ is right core invertible.
\end{proof}

The following theorem shows that $p+a^{n}\in \mathcal{R}^{-1}_{r}$ in Theorem \ref{the25} is true when taking $n \geq 2$.
Before it, we state one auxiliary result.

\begin{lemma}\cite[Exercise 1.6]{L2001}\label{lem26}
Let $a, b \in \mathcal{R}$. Then $1+ab$ is right invertible if and only if $1+ba$ is right invertible.
\end{lemma}

\begin{theorem}\label{the27}
Let $a\in \mathcal{R}$ and $n \geq 2$. Then the following statements are equivalent:

$(i)$ $a$ is right core invertible;

$(ii)$ there exists a unique projection $p$ such that $pa=0$ and $u=p+a^{n}\in \mathcal{R}^{-1}_{r}$;

$(iii)$ there exists a unique projection $p$ such that $pa=0$ and $w=a^{n}(1-p)+p\in \mathcal{R}^{-1}_{r}$.
\end{theorem}
\begin{proof}
$(i) \Rightarrow (ii)$ Set $p=1-aa^{\tiny{\textcircled{\tiny\#}}}_{r}$.
We observe first that $p$ is a projection satisfying $pa=0$.
From Theorem \ref{the25}, we get $a+1-aa^{\tiny{\textcircled{\tiny\#}}}_{r}=1+aa^{\tiny{\textcircled{\tiny\#}}}_{r}(a-1)$ is right invertible.
By Lemma \ref{lem26}, we have
$1+(a-1)aa^{\tiny{\textcircled{\tiny\#}}}_{r}=1+a^{2}a^{\tiny{\textcircled{\tiny\#}}}_{r}-aa^{\tiny{\textcircled{\tiny\#}}}_{r}$ is right invertible.
When $n=2$, it is easy to verify that
\begin{center}
$a^{2}+1-aa^{\tiny{\textcircled{\tiny\#}}}_{r}=(a^{2}a^{\tiny{\textcircled{\tiny\#}}}_{r}+1-aa^{\tiny{\textcircled{\tiny\#}}}_{r})(a+1-aa^{\tiny{\textcircled{\tiny\#}}}_{r})$.
\end{center}
is right invertible.
We assume that $n > 2$ and the result is true for the case $n-1$.
Hence, $a^{n}+1-aa^{\tiny{\textcircled{\tiny\#}}}_{r}
=(a^{2}a^{\tiny{\textcircled{\tiny\#}}}_{r} + 1-aa^{\tiny{\textcircled{\tiny\#}}}_{r})(a^{n-1}+1-aa^{\tiny{\textcircled{\tiny\#}}}_{r})$ is right invertible.
For the uniqueness of the projection, it is similar to $(i) \Rightarrow (ii)$ in Theorem \ref{the25}.

$(ii) \Rightarrow (i)$ By hypothesis $pa=0$ and $p+a^{n}\in \mathcal{R}^{-1}_{r}$, we have $(1-p)(p+a^{n})=a^{n}$ and $1-p=a^{n}(p+a^{n})^{-1}_{r}$.
Set $x=a^{n-1}(p+a^{n})^{-1}_{r}$. Then we obtain
\begin{center}
$ax=a^{n}(p+a^{n})^{-1}_{r}=1-p$,
\end{center}
it gives that $(ax)^{\ast}=ax$. It is easy to get $axa=(1-p)a=a$.
Moreover, $ax^{2}=(1-p)a^{n-1}(p+a^{n})^{-1}_{r}=a^{n-1}(p+a^{n})^{-1}_{r}=x$.
By Lemma \ref{lem21}, we obtain $a$ is right core invertible.

$(i) \Rightarrow (iii)$ Set $p=1-aa^{\tiny{\textcircled{\tiny\#}}}_{r}$.
It is clear $p$ is a projection satisfying $pa=0$.
Note that $1+aa^{\tiny{\textcircled{\tiny\#}}}_{r}(a^{n}-1)=1+a^{n}-aa^{\tiny{\textcircled{\tiny\#}}}_{r}$ is right invertible by $(i) \Rightarrow (ii)$.
Hence, $w=a^{n}aa^{\tiny{\textcircled{\tiny\#}}}_{r}+1-aa^{\tiny{\textcircled{\tiny\#}}}_{r}=1+(a^{n}-1)aa^{\tiny{\textcircled{\tiny\#}}}_{r}$
is right invertible by Lemma \ref{lem26}.

$(iii) \Rightarrow (i)$ Note that $(1-p)w=(1-p)[a^{n}(1-p)+p]=a^{n}(1-p)$. Then $1-p=a^{n}(1-p)w^{-1}_{r}$.
We set $x=a^{n-1}(1-p)w^{-1}_{r}$. It is clear that $ax=1-p=(ax)^{\ast}$. And $axa=(1-p)a=a$ and
$ax^{2}=(1-p)x=x$. Thus, by Lemma \ref{lem21}, we obtain $a$ is right core invertible.
\end{proof}

By Remark \ref{rem22}, we know that when $a$ is right  core invertible,
$a^{\pi}=1-aa^{\tiny{\textcircled{\tiny\#}}}_{r}$ is the spectral idempotents of $a$.
It is of interest to know whether two elements in the ring
have equal idempotents determined by right core inverses. In the following result, some characterizations of those elements with
equal corresponding idempotents are given.

\begin{proposition} \label{pro28}
Let $a, b\in \mathcal{R}^{\tiny{\textcircled{\tiny\#}}}_{r}$. Then the following are equivalent:

$(i)$ $aa^{\tiny{\textcircled{\tiny\#}}}_{r} = bb^{\tiny{\textcircled{\tiny\#}}}_{r}$;

$(ii)$ $a\mathcal{R}=b\mathcal{R}$;

$(iii)$ $a^{\pi}b = 0$ and $a^{\pi}+ b \in \mathcal{R}^{-1}_{r}$;

$(iv)$ $a^{\pi}b = 0$ and $a^{\pi} + b(1 - a^{\pi} ) \in \mathcal{R}^{-1}_{r}$.

In addition, if one of statements (i)--(iv) holds,
then $ab$ is right core invertible and $b^{\tiny\textcircled{\tiny\#}}_ra^{\tiny\textcircled{\tiny\#}}_r$ is a right core inverse of $ab$.
\end{proposition}

\begin{proof} $(i) \Rightarrow (ii)$ From $aa^{\tiny{\textcircled{\tiny\#}}}_{r} = bb^{\tiny{\textcircled{\tiny\#}}}_{r}$,
we get $a=bb^{\tiny{\textcircled{\tiny\#}}}_{r}a$ and $b=aa^{\tiny{\textcircled{\tiny\#}}}_{r}b$ which imply $a\mathcal{R}=b\mathcal{R}$.

$(ii) \Rightarrow (i)$ The assumption $a\mathcal{R}=b\mathcal{R}$ gives $a=bb^{\tiny{\textcircled{\tiny\#}}}_{r}a$ and $b=aa^{\tiny{\textcircled{\tiny\#}}}_{r}b$.
Thus, $$aa^{\tiny{\textcircled{\tiny\#}}}_{r} = bb^{\tiny{\textcircled{\tiny\#}}}_{r}aa^{\tiny{\textcircled{\tiny\#}}}_{r}
= (aa^{\tiny{\textcircled{\tiny\#}}}_{r} bb^{\tiny{\textcircled{\tiny\#}}}_{r})^*
= (bb^{\tiny{\textcircled{\tiny\#}}}_{r})^*
=bb^{\tiny{\textcircled{\tiny\#}}}_{r}.$$

$(i) \Leftrightarrow (iii)$ and $(i) \Leftrightarrow (iv)$ These equivalences follow by Theorem \ref{the25}.

The equality $aa^{\tiny\textcircled{\tiny\#}}_r=bb^{\tiny\textcircled{\tiny\#}}_r$ gives
$abb^{\tiny\textcircled{\tiny\#}}_ra^{\tiny\textcircled{\tiny\#}}_r
=a^2(a^{\tiny\textcircled{\tiny\#}}_r)^2=aa^{\tiny\textcircled{\tiny\#}}_r$,
$abb^{\tiny\textcircled{\tiny\#}}_ra^{\tiny\textcircled{\tiny\#}}_rab=aa^{\tiny\textcircled{\tiny\#}}_rab=ab$
and $ab(b^{\tiny\textcircled{\tiny\#}}_ra^{\tiny\textcircled{\tiny\#}}_r)^2=
aa^{\tiny\textcircled{\tiny\#}}_rb^{\tiny\textcircled{\tiny\#}}_ra^{\tiny\textcircled{\tiny\#}}_r
=bb^{\tiny\textcircled{\tiny\#}}_rb^{\tiny\textcircled{\tiny\#}}_ra^{\tiny\textcircled{\tiny\#}}_r=
b^{\tiny\textcircled{\tiny\#}}_ra^{\tiny\textcircled{\tiny\#}}_r$. So, $ab$ is right core invertible and
$(ab)^{\tiny\textcircled{\tiny\#}}_r=b^{\tiny\textcircled{\tiny\#}}_ra^{\tiny\textcircled{\tiny\#}}_r$.
\end{proof}

More sufficient conditions for the reverse order law of right core invertible elements are present now.

\begin{proposition}
Let $a, b\in \mathcal{R}$ be right core invertible. Then the following statements are equivalent:

$(i)$ $a=abb^{\tiny\textcircled{\tiny\#}}_r$ and $b=aa^{\tiny\textcircled{\tiny\#}}_rb$;

$(ii)$ $a^{\ast}\mathcal{R}\subseteq b\mathcal{R}\subseteq a\mathcal{R}$.

In addition, if one of statements (i)--(ii) holds,
then $ab$ is right core invertible and $b^{\tiny\textcircled{\tiny\#}}_ra^{\tiny\textcircled{\tiny\#}}_r$ is a right core inverse of $ab$.
\end{proposition}

\begin{proof} $(i)$ $\Rightarrow$ $(ii)$ The hypothesis $b=aa^{\tiny\textcircled{\tiny\#}}_rb$ yields $b\mathcal{R}\subseteq a\mathcal{R}$.
Applying involution to $a=abb^{\tiny\textcircled{\tiny\#}}_r$, we get $a^*=bb^{\tiny\textcircled{\tiny\#}}_ra^*$, that is,
$a^{\ast}\mathcal{R}\subseteq b\mathcal{R}$.

$(ii)$ $\Rightarrow$ $(i)$ From $a^{\ast}\mathcal{R}\subseteq b\mathcal{R}$ and $b\mathcal{R}\subseteq a\mathcal{R}$,
we obtain $a^*=bb^{\tiny\textcircled{\tiny\#}}_ra^*$ and
$b=aa^{\tiny\textcircled{\tiny\#}}_rb$. Thus, $a=(a^*)^*=(bb^{\tiny\textcircled{\tiny\#}}_ra^*)^*=abb^{\tiny\textcircled{\tiny\#}}_r$.

Using $a=abb^{\tiny\textcircled{\tiny\#}}_r$, we conclude that
$abb^{\tiny\textcircled{\tiny\#}}_ra^{\tiny\textcircled{\tiny\#}}_r
=aa^{\tiny\textcircled{\tiny\#}}_r$ and
$abb^{\tiny\textcircled{\tiny\#}}_ra^{\tiny\textcircled{\tiny\#}}_rab=aa^{\tiny\textcircled{\tiny\#}}_rab=ab$.
By $b=aa^{\tiny\textcircled{\tiny\#}}_rb$, we get
$b^{\tiny\textcircled{\tiny\#}}_r=b(b^{\tiny\textcircled{\tiny\#}}_r)^2=aa^{\tiny\textcircled{\tiny\#}}_rb
(b^{\tiny\textcircled{\tiny\#}}_r)^2=aa^{\tiny\textcircled{\tiny\#}}_rb^{\tiny\textcircled{\tiny\#}}_r$. Therefore,
$ab(b^{\tiny\textcircled{\tiny\#}}_ra^{\tiny\textcircled{\tiny\#}}_r)^2=
aa^{\tiny\textcircled{\tiny\#}}_rb^{\tiny\textcircled{\tiny\#}}_ra^{\tiny\textcircled{\tiny\#}}_r=
b^{\tiny\textcircled{\tiny\#}}_ra^{\tiny\textcircled{\tiny\#}}_r$ and
$(ab)^{\tiny\textcircled{\tiny\#}}_r=b^{\tiny\textcircled{\tiny\#}}_ra^{\tiny\textcircled{\tiny\#}}_r$.
\end{proof}

Let $p=p^2\in\mathcal{R}$ be an idempotent. Then we can represent any element
$a\in\mathcal{R}$ as $$a=\sbmatrix{cc}
a_{11}&a_{12}\\a_{21}&a_{22}\endsbmatrix_p,$$ where $a_{11}=pap$,
$a_{12}=pa(1-p)$, $a_{21} = (1-p)ap$, $a_{22}= (1-p)a(1-p)$.

If $p=p^2=p^*$, then
$$a^*=\sbmatrix{cc}
a_{11}^*&a_{21}^*\\a_{12}^*&a_{22}^*\endsbmatrix_p.$$
Now we give matrix representations for a right core invertible element and its right core inverse.

\begin{theorem} \label{the210}
Let $a\in\mathcal{R}$. Then the following statements are equivalent:

$(i)$ $a$ is right core invertible and $x\in\mathcal{R}$ is a right core inverse of $a$;

$(ii)$ there exists a projection $q\in\mathcal{R}$ such that
\begin{equation}
a=\sbmatrix{cc} a_1&a_2\\0&0\endsbmatrix_q,\qquad
x=\sbmatrix{cc} x_1&x_2\\0&0\endsbmatrix_q,\label{a-x-right-core}
\end{equation}
where $a_1$ is right invertible in $q\mathcal{R}q$, $x_1=(a_1)^{-1}_r$ and $a_1x_2=0$;

$(iii)$ there exists a projection $p\in\mathcal{R}$ such that
$$a=\sbmatrix{cc} 0&0\\a_1&a_2\endsbmatrix_p,\qquad
x=\sbmatrix{cc} 0&0\\x_1&x_2\endsbmatrix_p,$$
where $a_2$ is right invertible in $(1-p)R(1-p)$, $x_2=(a_2)^{-1}_r$ and $a_2x_1=0$.
\end{theorem}

\begin{proof}
$(i) \Rightarrow (ii)$ Let $q=ax$. Then $qa=axa=a$ implies $(1-q)a=0$ and $qx=ax^2=x$.
So, $a$ and $x$ are represented as in (\ref{a-x-right-core}). Since $a_1=qaq=aq=a^2x$ and $x_1=qxq=xax$,
we get $a_1x_1=ax=q$, i.e. $x_1$ is a right inverse of $a_1$ in $qRq$. By $x_2=qx(1-q)=x(1-ax)$, we have
$a_1x_2=ax(1-ax)=0$.

$(ii)\Rightarrow (i)$ Because $ax=\sbmatrix{cc} q&0\\0&0\endsbmatrix_q$,
we can verify that $x$ satisfies $(ax)^{\ast}=ax$, $axa=a$ and $x=ax^{2}$.
Using Lemma \ref{lem21}, we deduce that $a$ is right core invertible and $x$ is a right core inverse of $a$.

$(i)\Leftrightarrow (iii)$ This equivalence follows similarly as $(i)\Leftrightarrow (ii)$ for $p=1-ax$.
\end{proof}

Notice that $p$ and $q$, which appear in Theorem \ref{the210}, are invariant on the choice of $x$.
We present one decomposition of a right core invertible element which is also invariant on the choice of right core inverse.

\begin{theorem}
Let $a\in\mathcal{R}$ be right core invertible. Then $a=a_1+a_2$, where

$(i)$ $a_1$ is right core invertible,

$(ii)$ $a_2^2=0$,

$(iii)$ $a_1a_2^*=0=a_2a_1$.

In addition, $a^2a^{\tiny\textcircled{\tiny\#}}_r$ is right core invertible and
$a^{\tiny\textcircled{\tiny\#}}_raa^{\tiny\textcircled{\tiny\#}}_r$ (or $a^{\tiny\textcircled{\tiny\#}}_r$) is a right core inverse of $a^2a^{\tiny\textcircled{\tiny\#}}_r$.
\end{theorem}

\begin{proof} If $x$ is a right core inverse of $a$, $a_1=a^2x$ and $a_2=a-a^2x=a(1-ax)$, then $a=a_1+a_2$.
Notice that $a_1a_2^*=a^2x(1-ax)a^*=(a^2x-a^2x)a^*=0$, $a_2a_1=a(1-ax)a^2x=a(a^2x-a^2x)=0$ and
$a_2^2=a(1-ax)a(1-ax)=a(a-a)(1-ax)=0$.

Set $y=xax$. Since $a_1y=a^2x^2ax=axax=ax$, we obtain that $(a_1y)^*=a_1y$, $a_1ya_1=axa^2x=a^2x=a_1$ and $a_1y^2=ax^2ax=xax=y$.
Hence, $a_1$ is right core invertible and $y$ is a right core inverse of $a_1$. Similarly, we prove that $a^{\tiny\textcircled{\tiny\#}}_r$ is a right core inverse of $a^2a^{\tiny\textcircled{\tiny\#}}_r$.
\end{proof}

\section{ More characterizations of right core inverses}

From Lemma \ref{lem21} and Remark \ref{rem22},
we know that if $a$ is right core invertible and $x$ is one right core inverse of $a$,
then $a\in \mathcal{R}^{\{1,3\}}$. By $a=axa=a(ax^{2})a$, it gives that $a\mathcal{R}=a^{2}\mathcal{R}$.
In the following, we will prove that the converse is also true.
Moreover, one more thing that we want to remind is that, in \cite{C1},
the authors gave some characterizations when $a\in \mathcal{R}^{\dag}$ and $a\mathcal{R} = a^{2}\mathcal{R}$.
It is proven that $a\in \mathcal{R}^{\dag}$ and $a\mathcal{R} = a^{2}\mathcal{R}$ if and only if $(a^{\ast}a)^{k}$ is right invertible along $a$ for $k\in \mathbb{N}^{+}$.

\begin{theorem}\label{the31}
Let $a \in \mathcal{R}$. Then the following are equivalent:

$(i)$ $a$ is right core invertible;

$(ii)$ $a\in \mathcal{R}^{\{1,3\}}$ and $a\mathcal{R}=a^{2}\mathcal{R}$;

$(iii)$ $a\in \mathcal{R}^{\{1,3\}}$ and $a\mathcal{R}=a^{n}\mathcal{R}$ for any $n\geq 2$;

$(iv)$ $a\in \mathcal{R}^{\{1,3\}}$ and $a\mathcal{R}=a^{n}\mathcal{R}$ for some $n\geq 2$;

$(v)$ $a^{n}$  is right core invertible and $a\mathcal{R}=a^{n}\mathcal{R}$ for any $n\geq 2$;

$(vi)$ $a^{n}$  is right core invertible and $a\mathcal{R}=a^{n}\mathcal{R}$ for some $n\geq 2$;

$(vii)$ $a^n\in \mathcal{R}^{\{1,3\}}$ and $a{R}=a^{k}\mathcal{R}$ for any $n\geq 2$ and $k > n$.

$(viii)$ $a^n\in \mathcal{R}^{\{1,3\}}$ and $a{R}=a^{k}\mathcal{R}$ for some $n\geq 2$ and $k > n$.

In this case, for any $n\geq 2$,
$$(a^n)^{\tiny{\textcircled{\tiny\#}}}_{r}=(a^{\tiny{\textcircled{\tiny\#}}}_{r})^n\quad and\quad
a^{\tiny{\textcircled{\tiny\#}}}_{r}=a^{n-1}(a^n)^{\tiny{\textcircled{\tiny\#}}}_{r}.$$
\end{theorem}

\begin{proof}
$(i) \Rightarrow (ii)$ and $(iii) \Rightarrow (iv)$ It is clear.

$(ii) \Rightarrow (iii)$ Since $a\mathcal{R}=a^{2}\mathcal{R}$, we have $a=a^{2}t$ for $t\in \mathcal{R}$. Then we get
\begin{center}
$a=aat=a(a^{2}t)t=a^{3}t^{2}=a^{2}at^{2}=a^{2}(a^{2}t)t^{2}=a^{4}t^{3}= \cdots =a^{n}t^{n-1}$.
\end{center}

$(iv) \Rightarrow (i)$ The condition $a\in \mathcal{R}^{\{1,3\}}$ implies that $p=1-aa^{(1,3)}$ is a projection.
Then it is clear that $pa=0$. Since $a\mathcal{R}=a^{n}\mathcal{R}$, we have $a=a^{n}t$ for some $t\in \mathcal{R}$.
Next we will prove that $a+p\in \mathcal{R}^{-1}_{r}$ and then, by Theorem \ref{the25}, we get $a$ is right core invertible.
Indeed, $(p+a)(1+a^{n-1}ta^{(1,3)}-a^{n-1}t)=p+a+a^{n}ta^{(1,3)}-a^{n}t=p+aa^{(1,3)}=1$.

$(i) \Rightarrow (v)$ Suppose that $x$ is a right core inverse of $a$ and $n\geq 2$.
Then, from
\begin{center}
$ax=a(ax^2)=a^2x^2=\dots=a^nx^n$,
\end{center}
we get
$(a^nx^n)^{\ast}=(ax)^{\ast}=ax=a^nx^n$.
Moreover, it is easy to get $a^nx^na^n=axa^n=a^n$ and $a^n(x^n)^2=axx^n=(ax^2)x^{n-1}=x^n$.
Hence, by Lemma \ref{lem21}, we obtain $a^n$  is right core invertible and
$(a^n)^{\tiny{\textcircled{\tiny\#}}}_{r}=x^n$.
Since $(i)$ is equivalent to $(iii)$, we have that $a\mathcal{R}=a^{n}\mathcal{R}$.

$(v) \Rightarrow (vi)$ This is obvious.

$(vi) \Rightarrow (i)$ Let $x=a^{n-1}(a^n)^{\tiny{\textcircled{\tiny\#}}}_{r}$ and $a=a^nt$ for some $t\in \mathcal{R}$.
Firstly, we observe that $ax=a^n(a^n)^{\tiny{\textcircled{\tiny\#}}}_{r}$, it gives $(ax)^{\ast}=ax$. Further,
$axa=a^n(a^n)^{\tiny{\textcircled{\tiny\#}}}_{r}a=a^n(a^n)^{\tiny{\textcircled{\tiny\#}}}_{r}a^nt
=a^nt=a$ and $ax^2=a^n(a^n)^{\tiny{\textcircled{\tiny\#}}}_{r}a^{n-1}(a^n)^{\tiny{\textcircled{\tiny\#}}}_{r}
=a^n(a^n)^{\tiny{\textcircled{\tiny\#}}}_{r}a^nta^{n-2}(a^n)^{\tiny{\textcircled{\tiny\#}}}_{r}
=a^nta^{n-2}(a^n)^{\tiny{\textcircled{\tiny\#}}}_{r}=a^{n-1}(a^n)^{\tiny{\textcircled{\tiny\#}}}_{r}=x$.
So, by Lemma \ref{lem21}, $a$ is right core invertible and $a^{\tiny{\textcircled{\tiny\#}}}_{r}=x$.

$(i) \Rightarrow (vii)$ Consequently, by previous proofs. 

$(vii) \Rightarrow(viii)$ is obvious.

$(viii) \Rightarrow(i)$ Let $x=a^{n-1}(a^n)^{(1,3)}$ and $a=a^kt$ for some $t\in \mathcal{R}$.
It is clear that $(ax)^{\ast}=ax=a^{n}(a^n)^{(1,3)}$. Further, by $k>n$,
$axa=a^n(a^n)^{(1,3)}a=a^n(a^n)^{(1,3)}a^kt
=a^kt=a$ and $ax^2=a^n(a^n)^{(1,3)}a^{n-1}(a^n)^{(1,3)}
=a^n(a^n)^{(1,3)}a^kta^{n-2}(a^n)^{(1,3)}
=a^kta^{n-2}(a^n)^{(1,3)}=a^{n-1}(a^n)^{(1,3)}=x$.
So, by Lemma \ref{lem21}, $a$ is right core invertible and $a^{\tiny{\textcircled{\tiny\#}}}_{r}=x$.
\end{proof}

\begin{remark}\label{rem32}
It is well known that $a\in \mathcal{R}^{\{1,3\}} \Leftrightarrow \mathcal{R}a=\mathcal{R}a^{\ast}a$.
Indeed, if $a\in \mathcal{R}^{\{1,3\}}$, then $a=aa^{(1,3)}a=(a^{(1,3)})^{\ast}a^{\ast}a$,
it gives that $\mathcal{R}a=\mathcal{R}a^{\ast}a$.
Conversely, if $a=ta^{\ast}a$ for some $t\in \mathcal{R}$, then $at^{\ast}=ta^{\ast}at^{\ast}$,
it gives that
$(at^{\ast})^{\ast}=(ta^{\ast}at^{\ast})^{\ast}=ta^{\ast}at^{\ast}=at^{\ast}$.
Moreover, $a=ta^{\ast}a$ implies that $a^{\ast}=a^{\ast}at^{\ast}$.
Postmultiply by $a$, we have $a^{\ast}a=a^{\ast}at^{\ast}a$.
Premultiply by $t$, we have $ta^{\ast}a=ta^{\ast}at^{\ast}a$, i.e. $a=at^{\ast}a$.
Then we obtain $t^{\ast}$ is a $\{1,3\}$-inverse of $a$, and $a\in \mathcal{R}^{\{1,3\}}$.
\end{remark}

In the following, we will give some characterizations of $\mathcal{R}a=\mathcal{R}(a^{\ast})^{n}a$, where $n\geq2$.
For the case of $n=2$, we can check that if $a=t(a^{\ast})^{2}a$ for some $t\in \mathcal{R}$, then $a^{\ast}=a^{\ast}a^{2}t^{\ast}$.
It gives that $t(a^{\ast})^{2}=ta^{\ast}a^{\ast}=ta^{\ast}(a^{\ast}a^{2}t^{\ast})=(ta^{\ast}a^{\ast}a)at^{\ast}=a^{2}t^{\ast}$.
This implies that $a=a^{2}t^{\ast}a$.
Hence, $\mathcal{R}a=\mathcal{R}(a^{\ast})^{2}a$ gives that $\mathcal{R}a=\mathcal{R}a^{\ast}a$ and $a\mathcal{R}=a^{2}\mathcal{R}$.
This means $\mathcal{R}a=\mathcal{R}(a^{\ast})^{2}a$ can implies $a$ is right core invertible.

\begin{corollary}\label{cor33}
Let $a \in \mathcal{R}$ and $n \geq 2$. Then the following are equivalent:

$(i)$ $a$ is right core invertible;

$(ii)$ $\mathcal{R}a=\mathcal{R}a^{\ast}a$ and $a\mathcal{R}=a^{n}\mathcal{R}$;

$(iii)$ $\mathcal{R}a=\mathcal{R}(a^{\ast})^{n}a$;

$(iv)$ $\mathcal{R}a^n=\mathcal{R}(a^{\ast})^na^n$ and $a\mathcal{R}=a^{k}\mathcal{R}$ for $k>n$;
\end{corollary}
\begin{proof}
$(i) \Leftrightarrow (ii)$ By Theorem \ref{the31} and Remark \ref{rem32}.

$(ii) \Rightarrow (iii)$ By hypothesis $a\mathcal{R}=a^{n}\mathcal{R}$, we have $a=a^{n}t$ for some $t\in \mathcal{R}$.
Then $a^{\ast}=t^{\ast}(a^{\ast})^{n}$. Since $\mathcal{R}a=\mathcal{R}a^{\ast}a$, we get $a=sa^{\ast}a$ for some $s\in \mathcal{R}$, and
$a=st^{\ast}(a^{\ast})^{n}a$. This means that $\mathcal{R}a=\mathcal{R}(a^{\ast})^{n}a$.

$(iii) \Rightarrow (ii)$ Using assumptions, we have $a=t(a^{\ast})^{n}a$ for some $t\in \mathcal{R}$,
then $a^{\ast}=a^{\ast}a^{n}t^{\ast}$.
It gives that $t(a^{\ast})^{n}=t(a^{\ast})^{n-1}a^{\ast}=t(a^{\ast})^{n-1}(a^{\ast}a^{n}t^{\ast})=[t(a^{\ast})^{n}a]a^{n-1}t^{\ast}=a^{n}t^{\ast}$.
This implies that $a=a^{n}t^{\ast}a$.
Hence, $\mathcal{R}a=\mathcal{R}(a^{\ast})^{n}a$ implies that $\mathcal{R}a=\mathcal{R}a^{\ast}a$ and $a\mathcal{R}=a^{n}\mathcal{R}$.

$(i) \Rightarrow (iv)$ By Theorem \ref{the31} and previous equivalences.
\end{proof}

Note that $a^{\ast}$ is left $(a, a^{\ast})$ invertible, if $\mathcal{R}a=\mathcal{R}(a^{\ast})^{2}a$.
Then by Corollary \ref{cor33}, we can obtain the following result.

\begin{corollary} Let $a\in R$. Then
$a$ is right $(a, a^{\ast})$ invertible if and only if $a^{\ast}$ is left $(a, a^{\ast})$ invertible.
\end{corollary}

\begin{remark}
We find $\mathcal{R}a=\mathcal{R}(a^{\ast})^{n}a$ for any $n \geq 2$, it implies that
$a^{\ast}$ is left $(a, a^{\ast})$ invertible if and only if $(a^{\ast})^{n}$ is left $(a, a^{\ast})$ invertible.
In fact, it is easy to obtain that $a$ is right $(a, a^{\ast})$ invertible if and only if $a^{n}$ is right $(a, a^{\ast})$ invertible.
\end{remark}

\section{ The related generalized core inverses}

In this section, some new characterizations of (generalized) core inverses are given.
Through these characterizations, we can clearly find the relationship between these generalized inverses.
Moreover, we will provide a relation schema of several kinds of (generalized one-sided) core inverses.

\begin{theorem}\label{the41}
Let  $a\in \mathcal{R}$ and $k \geq 1$. Then the following are equivalent:

$(i)$  $a$ is core invertible;

$(ii)$ there exists $x\in \mathcal{R}$ such that $xa^{2}=a$, $x^{k}=ax^{k+1}$ and $(a^kx^k)^{\ast}=a^kx^k$;

$(iii)$ there exists $x\in\mathcal{R}$ such that $xa^{2}=a$, $x^{k}=ax^{k+1}$ and $(ax)^{\ast}=ax$.
\end{theorem}

\begin{proof} If $k=1$, then it is clear that $(i)\Leftrightarrow(ii)\Leftrightarrow(iii).$ Let us assume that $k\geq 2$ in the rest of the proof.

$(i)\Rightarrow(ii)$-$(iii)$
Let $x=a^{\tiny{\textcircled{\tiny \#}}}$.
Then we obtain $xa^{2}=a$, $x=ax^{2}$ and $(ax)^{\ast}=ax$.
For $k\geq 2$, it is easy to get $ax^{k+1}=ax^{2}x^{k-1}=x^{k}$.
Note that $ax=a(ax^2)=a^2x^2=a^{2}(ax^{2})x=\dots=a^kx^k$. Hence, $(a^kx^k)^{\ast}=a^kx^k$.

$(ii)\Rightarrow(i)$ Set $z=a^{k-1}x^k$. Then $za^2=a^{k-1}x^ka^2=a^{k-1}x^{k-1}(xa)a=a^{k-1}x^{k-1}(x^ka^{k})a=a^{k-1}x^{2k-1}a^{k+1}=x^{k}a^{k+1}=a$ and $az^2=aa^{k-1}x^ka^{k-1}x^k=a^{k}x^ka^{k-1}x^k=a^{k}x^{2k}a^{2k-1}x^k=x^{k}a^{2k-1}x^k=(x^{k}a^{k})a^{k-1}x^k=xaa^{k-1}x^k=xa^kx^k=a^{k-1}x^k=z$. Note that
$az=a^kx^k$, thus $(az)^*=az$. Hence $a^{\tiny{\textcircled{\tiny \#}}}=z=a^{k-1}x^k$.

$(iii)\Rightarrow(i)$  Note that $a=xa^{2}=x(xa^{2})a=x^{2}a^{3}=x^{2}(xa^{2})a^{2}=x^{3}a^{4}=\cdots=x^{n}a^{n+1}$.
Write $z=xax$. Then $az=axax=ax(x^{k}a^{k+1})x=(ax^{k+1})a^{k+1}x=x^{k}a^{k+1}x=ax$, it gives that $(az)^{\ast}=az$.
It is easy to get $za^{2}=xa(xa^{2})=xa^{2}=a$.
Moreover, $az^{2}=axxax=axx(x^{k}a^{k+1})x=ax^{k+1}xa^{k+1}x=x^{k+1}a^{k+1}x=xax=z$.
This implies that $a^{\tiny{\textcircled{\tiny \#}}}=xax$.
\end{proof}

\begin{proposition}\label{prop42}
Let  $a, x\in \mathcal{R}$ and $k \geq 1$. Then the following are equivalent:

$(i)$ $x$ is the core inverse of $a$;

$(ii)$ $xa^{2}=a$,  $xax=x$, $(ax)^{\ast}=ax$ and $x^{k}=ax^{k+1}$;

$(iii)$ $xa^{2}=a$,  $x^{k+1}a^{k+1}x=x$, $(x^{k}a^{k+1}x)^{\ast}=x^{k}a^{k+1}x$ and $x^{k}=x^ka^{k+1}x^{k+1}$.
\end{proposition}
\begin{proof}
$(i)\Rightarrow(ii)$ It is clear.

$(ii)\Rightarrow(i)$ It suffices to prove $ax^2=x$. Indeed,
by $xa^{2}=a$, we get $xa=x(xa^{2})=x^{2}a^{2}=\cdots=x^{n}a^{n}$,
thus, $x=xax=x^ka^kx=ax^{k+1}a^kx=ax(x^{k}a^k)x=axxax=ax^2$.

$(ii)\Leftrightarrow(iii)$ Because $a=xa^{2}$ implies $a=x^ka^{k+1}$, this equivalence is obvious.
\end{proof}

In the following result, we will change the condition $(ax)^{\ast}=ax$ in Theorem \ref{the41} into $(xa)^{\ast}=xa$.
It is interesting to find that $a$ is not only core invertible but also EP.

\begin{theorem} \label{the43}
Let $a\in \mathcal{R}$ and $k \geq 1$. Then the following are equivalent:

$(i)$ $a$ is EP;

$(ii)$ there exists $x\in \mathcal{R}$ such that $xa^{2}=a$, $(xa)^{\ast}=xa$ and $x^{k}=ax^{k+1}$;

$(ii)$ there exists $x\in \mathcal{R}$ such that $xa^{2}=a$, $(x^{k+1}a^{k+1})^{\ast}=x^{k+1}a^{k+1}$ and $x^{k}=x^{k}a^{k+1}x^{k+1}$.
\end{theorem}

\begin{proof}
$(i)\Rightarrow(ii)$ If $a$ is EP, then $a \in \mathcal{R}^{\dag} \cap \mathcal{R}^{\sharp}$ with $a^{\sharp} = a^{\dag}$. Take $x=a^{\dag}=a^{\sharp}$.

$(ii)\Rightarrow(i)$ From $xa^2=a$ and $a=xa^2=x^ka^{k+1}=ax^{k+1}a^{k+1}=ax^{k}a^{k}=a^2x^{k+1}a^{k}=a^2x^2a$, it follows that $a^{\sharp}$ exists with $a^{\sharp}=x^2a$. Since $aa^{\sharp}=ax^2a=ax^{k+1}a^k=x^ka^k=xa$. Thus, $(aa^{\sharp})^{\ast}=aa^{\sharp}$. Hence, $a$ is EP.

$(ii)\Leftrightarrow(iii)$ This part is clear.
\end{proof}

In \cite{GC}, Gao and Chen introduced the concept of pseudo core inverse in $\ast$-rings.
\begin{definition}\emph{\cite{GC}}
Let $a\in R$. The pseudo core inverse  of $a$, denoted by $a^{\scriptsize{\textcircled{\tiny D}}}$, is the unique solution to  system
\begin{center}
$xa^{k+1} = a^{k}$ for some  $k\geq 1$, $ax^{2} = x$ and  $(ax)^{\ast} = ax$.
\end{center}
\end{definition}
The smallest positive integer $k$ satisfying above equations is called the pseudo core index of $a$. If $a$ is  pseudo core invertible, then it is Drazin invertible and its pseudo core index  equals to index, see  \cite{GC}. Here and subsequently,  the pseudo core index  is denoted by $\text{ind}(a)$.

In the following, we propose the notion of one-sided pseudo core inverse in a $\ast$-ring $\mathcal{R}$ as a generalized one-sided core inverse, after which,
some characterizations of right pseudo core inverses are given.

\begin{definition}
Let  $a\in \mathcal{R}$. Then
$a$ is called right pseudo core invertible if there exist $x\in \mathcal{R}$ and positive integer $k$ such that $axa^k=a^{k}$, $x=ax^{2}$ and $(ax)^{\ast}=ax$.
\end{definition}
We use the symbol $a^{\scriptsize{\textcircled{\tiny D}}}_{r}$ to denote the right pseudo core inverse of $a$, when $a$ is right pseudo core invertible. Before we characterize right pseudo core invertible elements, a characterization of pseudo core invertible elements is presented.

\begin{lemma}\emph{\cite[Theorem 2.3]{GC}}\label{GC}
Let  $a\in \mathcal{R}$ and $k \geq 1$. Then the following are equivalent:

$(i)$ $a$ is pseudo core invertible with pseudo core index $k$;

$(ii)$ $a^{k} \in \mathcal{R}^{\{1,3\}}$ and $a \in \mathcal{R}^{D}$ with index $k$.

In this case, $a^{\scriptsize{\textcircled{\tiny D}}}=a^Da^k(a^k)^{(1,3)}.$
\end{lemma}

\begin{theorem} Let $a\in \mathcal{R}$ and $k \geq 1$. Then the following are equivalent:

$(i)$ $a$ is  pseudo core invertible;

$(ii)$ there exists $x\in\mathcal{R}$ such that $xa^{k+1}=a^k$, $ax^{k+1}=x^k$ and $(a^kx^{k})^*=a^kx^{k}$;

$(iii)$ there exists $x\in\mathcal{R}$ such that $a^kx^{k+1}a^{k+1}=a^k$, $ax^{2}=x$ and $(a^{k+1}x^{k+1})^*=a^{k+1}x^{k+1}$.
\end{theorem}

\begin{proof} $(i)\Rightarrow(ii)$ It is clear.

$(ii)\Rightarrow(i)$  Note that $a^k\in \mathcal{R}a^{k+1}$ and $a^k=xa^{k+1}=x^ka^{2k}=ax^{k+1}a^{2k}=a^{k+1}x^{2k+1}a^{2k}\in a^{k+1}\mathcal{R}$,
thus $a$ is Drazin invertible with $a^D=x^{k+1}a^k$.
Again note that $a^kx^ka^k=a^kx^{2k}a^{2k}=x^ka^{2k}=a^k$ and $(a^kx^k)^*=a^kx^k$ imply that $x^k$ is a \{1,3\}-inverse of $a^k$.
In view of Lemma \ref{GC}, $a$ is pseudo core invertible with $a^{\scriptsize{\textcircled{\tiny D}}}=a^Da^k(a^k)^{(1,3)}=x^{k+1}a^ka^kx^k=a^{k-1}x^{k}$.

$(ii)\Leftrightarrow(iii)$ The equality $x=ax^{2}$ gives $x=a^kx^{k+1}$ and so the rest is clear.
\end{proof}

In the following result, we will reveal the relationship between right pseudo core inverses and right core inverses.
\begin{theorem}
Let $a\in \mathcal{R}$. Then the following are equivalent:

$(i)$ $a$ is right pseudo core invertible;

$(ii)$ $a^k$ is right core invertible for some positive integer $k$.
\end{theorem}

\begin{proof}
$(i)\Rightarrow(ii)$ If $a$ is right pseudo core invertible, then we can check that $(a^{\scriptsize{\textcircled{\tiny D}}}_{r})^k$ is a right core inverse of $a^k$. Indeed, $a^k(a^{\scriptsize{\textcircled{\tiny D}}}_{r})^k=aa^{\scriptsize{\textcircled{\tiny D}}}_{r}$,
$a^k(a^{\scriptsize{\textcircled{\tiny D}}}_{r})^ka^k=aa^{\scriptsize{\textcircled{\tiny D}}}a^k=a^k$ and $a^k[(a^{\scriptsize{\textcircled{\tiny D}}}_{r})^k]^2=(a^{\scriptsize{\textcircled{\tiny D}}}_{r})^k$.

$(ii)\Rightarrow(i)$ If $a^k$ is right core invertible for some positive integer $k$, then we can check that $a^{k-1}(a^{k})^{\tiny{\textcircled{\tiny \#}}}_{r}$ is a right pseudo core inverse of $a$. Here we omit the details.
\end{proof}

In what follows, we give  characterizations of right pseudo core invertible elements.
\begin{theorem}\label{the47}
Let $a\in \mathcal{R}$. Then the following are equivalent:

$(i)$ $a$ is right pseudo core invertible;

$(ii)$ $a^{k+1}ya^k=a^k$ and $(a^{k+1}y)^*=a^{k+1}y$ for some $y\in \mathcal{R}$ and positive integer $k$;

$(iii)$ $a^{k}\in\mathcal{R}^{\{1,3\}}$ and $a^{k}\mathcal{R}=a^{k+1}\mathcal{R}$ for some positive integer $k$;

$(iv)$ $a^{k+1}x^{k+1}a^{k}=a^k$, $ax^{2}=x$ and $(a^{k+1}x^{k+1})^*=a^{k+1}x^{k+1}$ for some $x\in \mathcal{R}$ and positive integer $k$.
\end{theorem}

\begin{proof}
$(i)\Rightarrow(ii)$ If $a$ is right pseudo core invertible,
according to the definition of the right pseudo core inverse,
then there exists $x\in \mathcal{R}$ such that
\begin{center}
$axa^{k} = a^{k}$ for some positive integer $k$, $ax^{2}=x$ and $(ax)^{\ast}=ax$.
\end{center}
Note that $ax=a(ax^{2})=a^{2}x^{2}=\cdots=a^{k+1}x^{k+1}$, write $y=x^{k+1}$, then we have $a^{k+1}y=a^{k+1}x^{k+1}=ax$ and
$a^{k+1}ya^k=axa^k=a^k$. We thus have $a^{k+1}ya^k=a^k$ and $(a^{k+1}y)^*=a^{k+1}y$.

$(ii)\Rightarrow(iii)$ From $a^{k+1}ya^k=a^k$ and $(a^{k+1}y)^*=a^{k+1}y$,
it follows that $ay$ is a $\{1,3\}$-inverse of $a^k$ and $a^{k}\mathcal{R}=a^{k+1}\mathcal{R}$.

$(iii)\Rightarrow(i)$ Suppose $a^k=a^{k+1}z$ for some $z\in \mathcal{R}$. Let $(a^k)^{(1,3)}$ is a  $\{1,3\}$-inverse of $a^{k}$.
Write $x=a^kz(a^k)^{(1,3)}$.
It is easy to check $x$ is a right pseudo core inverse of $a$.

$(i)\Leftrightarrow(iv)$ This equivalence can be easily check.
\end{proof}

\begin{theorem}
Let $a\in \mathcal{R}$. Then the following are equivalent:

$(i)$ $a$ is right pseudo core invertible;

$(ii)$ $\mathcal{R}a^k=\mathcal{R}(a^k)^{\ast}a^k$ and $a^k\mathcal{R}=a^{k+1}\mathcal{R}$ for some positive integer $k$;

$(iii)$ $\mathcal{R}a^k=\mathcal{R}(a^{\ast})^{k+1}a^k$ for some positive integer $k$.
\end{theorem}

\begin{proof}
The proof is similar to Theorem \ref{the31}
\end{proof}

The matrix representations of right pseudo core invertible element and its right pseudo core inverse are presented in the following theorem.

\begin{theorem}
Let $a\in\mathcal{R}$. Then the following statements are equivalent:
\begin{itemize}
\item[(i)] $a$ is right pseudo core invertible and $x\in\mathcal{R}$ is a right pseudo core inverse of $a$;

\item[(ii)] there exists a projection $q\in\mathcal{R}$ such that
\begin{equation}a=\sbmatrix{cc} a_1&a_2\\a_3&a_4\endsbmatrix_q,\qquad
x=\sbmatrix{cc} x_1&x_2\\0&0\endsbmatrix_q,\label{a-x-right-ps-core}\end{equation}
where $a_1$ is right invertible in $qRq$, $x_1=(a_1)^{-1}_r$, $a_1x_2=0$, $a_3x_1=0$, $a_3x_2=0$ and $qa^k=a^k$ for some $k\geq 1$.
\end{itemize}
\end{theorem}

\begin{proof} $(i) \Rightarrow (ii)$ For $q=ax$, we get $qa^k=a^k$ and $qx=x$ implying (\ref{a-x-right-ps-core}).
Since $$\sbmatrix{cc} a_1x_1&a_1x_2\\a_3x_1&a_3x_2\endsbmatrix_q=ax=q=\sbmatrix{cc} q&0\\0&0\endsbmatrix_q,$$ the rest is clear.

$(ii) \Rightarrow (i)$ We prove this implication by elementary computations.
\end{proof}

\begin{remark}
In \cite{MDM}, the authors introduced the definition of weighted core inverse, called it $e$-core inverse of $a$, where $e$ is an invertible
Hermitian element. Similarly, the concept of one-sided $e$-core inverse of $a$ can be given:

Let $a\in \mathcal{R}$ and let $e\in \mathcal{R}$ be an invertible Hermitian element. Then
$a$ is right $($or left$)$ $e$--core invertible if $a$ is right $($left$)$ $(a,a^*e)$-invertible.

Recall that $a$ is right $e$--core invertible if and only if $a^*e\in a^*ea^2\mathcal{R}$
if and only if there exists $x\in \mathcal{R}$ such that $x\in a\mathcal{R}$ and $a^*eax=a^*e$.
Also, $a$ is left $e$--core invertible if and only if $a\in \mathcal{R}a^*ea^2$
if and only if there exists $x\in \mathcal{R}$ such that $x\in \mathcal{R}a^*e$ and $xa^2=a$.

It is not difficult to extend some characterizations of right core inverses to right $e$-core inverses.
Here we omit the details.
\end{remark}

Before the end of the section, we will provide the relation schema of several kinds of (generalized one-sided) core inverses.

\newpage

\includegraphics[width=6.5in]{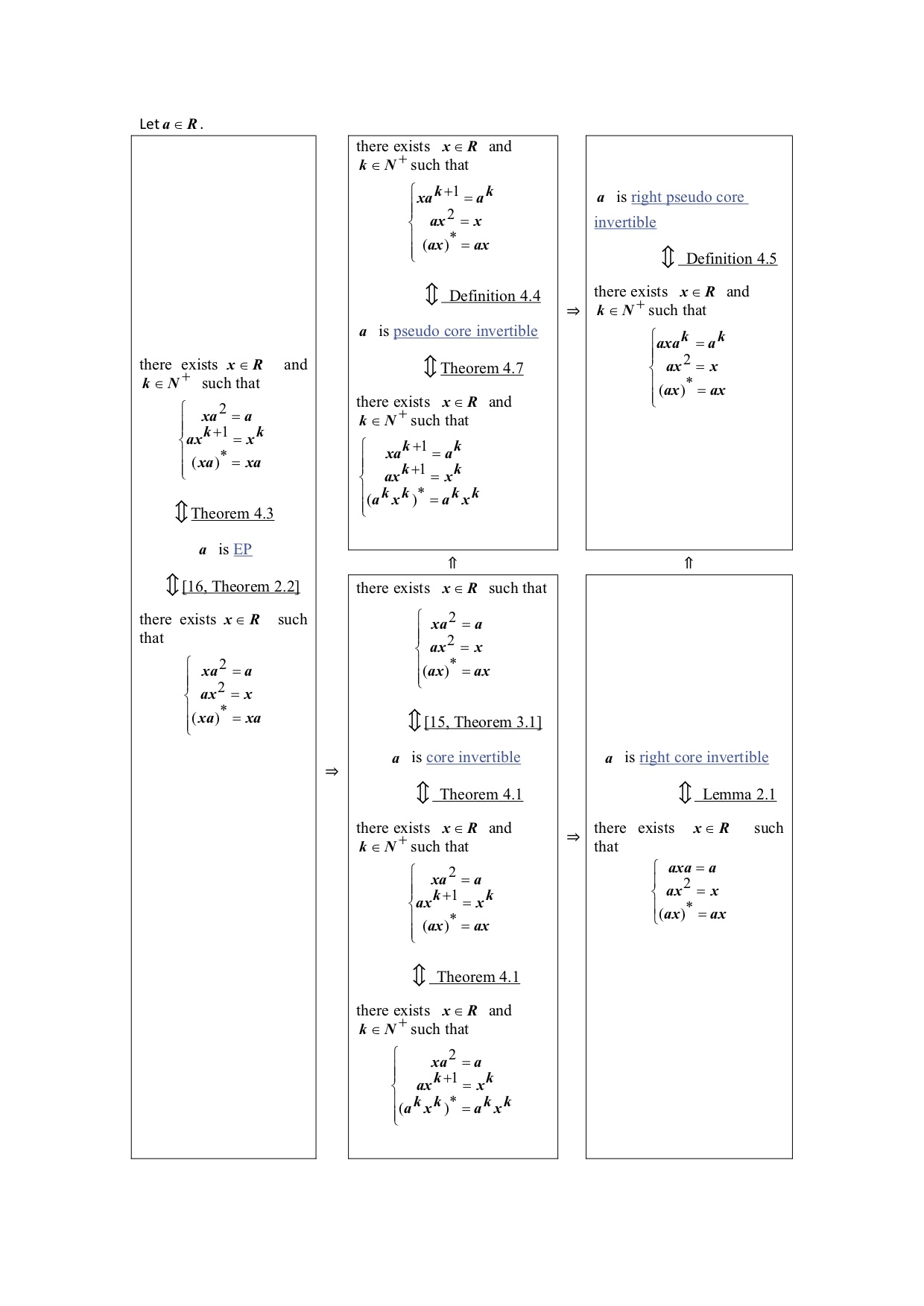}


\vspace{1mm}

\begin{thebibliography}{s2}

\bibitem{OMB1} Baksalary O.M., Trenkler G. (2010). Core inverse of matrices. Linear Multilinear Algebra 58:681-697.

\bibitem{C1} Chen J.L., Zou H.L., Zhu H.H., Patr\'{i}cio P. (2017). The one-sided inverse along an element in semigroups and rings. Mediterr. J. Math. 14:208-225.

\bibitem{MPD1} Drazin M. P. (2012). A class of outer generalized inverses. Linear Algebra Appl. 436:1909-1923.

\bibitem{MPD3} Drazin M. P. (1958). Pseudo-inverses in associative rings and semigroups. Am. Math. Mon. 65:506-514.

\bibitem{MPD2} Drazin M. P. (2016). Left and right generalized inverses. Linear Algebra Appl. 510:64-78.

\bibitem{GC} Gao Y.F., Chen J.L. (2018). Pseudo core inverses in rings with involution. Comm. Algebra 46(1):38-50.

\bibitem{L2001} Lam T.Y.  A First Course in Noncommutative Rings. Grad. Text in Math. Vol. 131. Springer-Verlag. Berlin-Heidelberg-New York, 2001.

\bibitem{LC} Li T. T., Chen J.L. (2018). Characterizations of core and dual core inverses in rings with involution.
Linear Multilinear Algebra 66:717-730.

\bibitem{MDM} Mosi\' c D., Deng C. Y., Ma H. F. (2018) On a weighted core inverse in a ring with involution. Comm. Algebra.
46:2332-2345.

\bibitem{R.P} Penrose R. (1955). A generalized inverse for matrices. Proc. Cambridge Philos. Soc. 51:406-413.

\bibitem{DSR1} Raki\'{c} D. S., Din\v{c}i\'{c} N. S., Djordjevi\'{c} D. S. (2014). Group, Moore-Penrose, core and dual core inverse in rings with involution.
Linear Algebra Appl. 463:115-133.

\bibitem{WCN} Wang L., Castro-Gonzalez N., Chen J. L. (2017). Characterizations of outer generalized inverses. Canad. Math. Bull.
60:861-871.

\bibitem{WD} Wang L., Mosi\' c D. (2018). Further results on left and right generalized inverses and their applications.  Submit.

\bibitem{S} Xu S. Z., Chen J. L., Bent\'{i}tez J. (2017). EP elements in rings with involution. ArXiv:1602.08184.

\bibitem{SZX1} Xu S. Z., Chen J. L., Zhang X. X. (2017). New characterizations for core inverses in rings with involution. Front. Math. China 12(1):231-246.

\end{thebibliography}
\end{document}